\documentclass{svmult}

\renewcommand{\email}[1]{\emailname: #1} 

\usepackage{mathptmx}       
\usepackage{helvet}         
\usepackage{courier}        

\usepackage{makeidx}         
\usepackage{graphicx}        
\usepackage[bottom]{footmisc}

\usepackage{latexsym}
\usepackage{amsmath}
\usepackage{amsfonts}
\usepackage{amssymb}
\usepackage{bm}

\usepackage{url}
\usepackage{algorithm}
\usepackage{algorithmic}
\usepackage[misc,geometry]{ifsym}

\renewenvironment{proof}{\noindent{\itshape Proof.}}{\smartqed\qed}

\spdefaulttheorem{assumption}{Assumption}{\upshape \bfseries}{\itshape}
\spdefaulttheorem{algo}{Algorithm}{\upshape \bfseries}{\itshape}

\DeclareSymbolFont{bbold}{U}{bbold}{m}{n}
\DeclareSymbolFontAlphabet{\mathbbold}{bbold}

\pdfoutput=1
\begin{document}
\title*{Exponential Sum Approximations for $t^{-\beta}$}
\titlerunning{Exponential Sum Approximations}
\author{William McLean}
\institute{William McLean (\Letter)
\at School of Mathematics and Statistics, 
The University of New South Wales, Sydney 2052, AUSTRALIA\\
\email{w.mclean@unsw.edu.au}}
\maketitle
\index{McLean, William}
\paragraph{Dedicated to Ian H.~Sloan on the occasion of his 80th birthday.}
\abstract{Given $\beta>0$~and $\delta>0$, the function~$t^{-\beta}$ 
may be approximated for~$t$ in a compact interval~$[\delta,T]$ by a 
sum of terms of the form~$w\E^{-at}$, with parameters $w>0$~and 
$a>0$. One such an approximation, studied by Beylkin~and
Monz\'on~\cite{BeylkinMonzon2010}, is obtained by applying the 
trapezoidal rule to an integral representation of~$t^{-\beta}$, after 
which Prony's method is applied to reduce the number of terms in 
the sum with essentially no loss of accuracy.  We review this method, 
and then describe a similar approach based on an alternative integral 
representation.  The main difference is that the new approach 
achieves much better results \emph{before} the application of Prony's 
method; after applying Prony's method the performance of both is much 
the same.}
\section{Introduction}
Consider a Volterra operator with a convolution kernel,
\begin{equation}\label{eq: Ku}
\mathcal{K}u(t)=(k*u)(t)=\int_0^t k(t-s)u(s)\,ds\quad\text{for $t>0$,}
\end{equation}
and suppose that we seek a numerical approximation to~$\mathcal{K}u$ 
at the points of a grid~$0=t_0<t_1<t_2<\cdots<t_{N_t}=T$.  For 
example, if we know~$U^n\approx u(t_n)$ and define (for simplicity) a 
piecewise-constant interpolant~$\tilde U(t)=U^n$ 
for~$t\in I_n=(t_{n-1},t_n)$, then
\[
\mathcal{K}u(t_n)\approx\mathcal{K}\tilde U(t_n)
	=\sum_{j=1}^n\omega_{nj}U^j\quad\text{where}\quad
	\omega_{nj}=\int_{I_j}k(t_n-s)\,ds.
\]
The number of operations required to compute this sum in the obvious 
way for~$1\le n\le N_t$ is proportional 
to~$\sum_{n=1}^{N_t}n\approx N_t^2/2$, and this quadratic growth can 
be prohibitive in applications where each $U^j$ is a large vector and 
not just a scalar.  Moreover, it might not be possible to store $U^j$ 
in active memory for all time levels~$j$.

These problems can be avoided using a simple, fast algorithm if the 
kernel~$k$ admits an exponential sum approximation
\begin{equation}\label{eq: general k(t)}
k(t)\approx\sum_{l=1}^L w_le^{b_lt}\quad\text{for $\delta\le t\le T$,}
\end{equation}
provided sufficient accuracy is achieved using only a moderate number 
of terms~$L$, for a choice of~$\delta>0$ that is smaller than the 
time step~$\Delta t_n=t_n-t_{n-1}$ for all~$n$.  Indeed, 
if~$\Delta t_n\ge\delta$ then $\delta\le t_n-s\le T$ 
for~$0\le s\le t_{n-1}$ so
\[
\sum_{j=1}^{n-1}\omega_{nj}U^j=\int_0^{t_{n-1}}k(t_n-s)\tilde U(s)\,ds
	\approx\int_0^{t_{n-1}}\sum_{l=1}^Lw_le^{b_l(t_n-s)}
		\tilde U(s)\,ds=\sum_{l=1}^L\Theta^n_l,
\]
where
\[
\Theta^n_l=w_l\int_0^{t_{n-1}}e^{b_l(t_n-s)}\tilde U(s)\,ds
	=\sum_{j=1}^{n-1}\kappa_{lnj}U^j
\quad\text{and}\quad
\kappa_{lnj}=w_l\int_{I_j}e^{b_l(t_n-s)}\,ds.
\]
Thus,
\begin{equation}\label{eq: fast KU}
\mathcal{K}\tilde U(t_n)\approx\omega_{nn}U^n+\sum_{l=1}^L\Theta^n_l,
\end{equation}
and by using the recursive formula
\[
\Theta^n_l=\kappa_{ln,n-1}U^{n-1}+e^{b_l\Delta t_n}\Theta^{n-1}_l
\quad\text{for $n\ge2$,}\quad\text{with $\Theta^1_l=0$,}
\]
we can evaluate $\mathcal{K}\tilde U(t_n)$ to an acceptable accuracy 
with a number of operations proportional to~$LN_t$ --- a substantial 
saving if $L\ll N_t$.  In addition, we may overwrite 
$\Theta^{n-1}_l$~with $\Theta^n_l$, and overwrite $U^{n-1}$ 
with~$U^n$, so that the active storage requirement is proportional 
to~$L$ instead of~$N_t$.

In the present work, we study two exponential sum approximations 
to the kernel~$k(t)=t^{-\beta}$ with~$\beta>0$.  Our starting point 
is the integral representation
\begin{equation}\label{eq: basic repn}
\frac{1}{t^\beta}=\frac{1}{\Gamma(\beta)}\int_0^\infty e^{-pt}p^\beta
	\,\frac{dp}{p}\quad\text{for $t>0$ and $\beta>0$,}
\end{equation}
which follows easily from the integral definition of the Gamma 
function via the substitution~$p=y/t$ (if $y$ is the original 
integration variable).  Section~\ref{sec: BM approx} discusses the 
results of Beylkin and Monz\'on~\cite{BeylkinMonzon2010}, who used the 
substitution~$p=e^x$ in~\eqref{eq: basic repn} to obtain
\begin{equation}\label{eq: old repn}
\frac{1}{t^\beta}=\frac{1}{\Gamma(\beta)}\int_{-\infty}^\infty
	\exp(-te^x+\beta x)\,dx.
\end{equation}
Applying the infinite trapezoidal rule with step size~$h>0$ leads to 
the approximation
\begin{equation}\label{eq: infinite trap}
\frac{1}{t^\beta}\approx\frac{1}{\Gamma(\beta)}\sum_{n=-\infty}^\infty
	w_ne^{-a_nt}
\end{equation}
where
\begin{equation}\label{eq: BM exponents and weights}
a_n=e^{hn}\quad\text{and}\quad w_n=he^{\beta nh}.
\end{equation}
We will see that the relative error,
\begin{equation}\label{eq: infinite rel err}
\rho(t)=1-\frac{t^\beta}{\Gamma(\beta)}\sum_{n=-\infty}^\infty
	w_ne^{-a_nt},
\end{equation}
satisfies a uniform bound for~$0<t<\infty$.  If $t$ is restricted to a 
compact interval~$[\delta,T]$ with~$0<\delta<T<\infty$, then we can 
similarly bound the relative error in the \emph{finite} exponential 
sum approximation
\begin{equation}\label{eq: finite trap}
\frac{1}{t^\beta}\approx\frac{1}{\Gamma(\beta)}\sum_{n=-M}^N
	w_ne^{-a_nt}\quad\text{for $\delta\le t\le T$,}
\end{equation}
for suitable choices of $M>0$~and $N>0$.

The exponents~$a_n=e^{nh}$ in the sum~\eqref{eq: finite trap} tend to 
zero as~$n\to-\infty$.  In Section~\ref{sec: Prony} we see how, for
a suitable threshold exponent size~$a^*$, Prony's method may be used
to replace $\sum_{a_n\le a^*}w_ne^{-a_nt}$ with an exponential sum 
having fewer terms.  This idea again follows Beylkin and 
Monz\'on~\cite{BeylkinMonzon2010}, who discussed it in the context of 
approximation by Gaussian sums.

\begin{figure}
\caption{Top: the integrand from~\eqref{eq: new repn} 
when~$\beta=1/2$ for different~$t$. Bottom: comparison between the
integrands from \eqref{eq: old repn}~and \eqref{eq: new repn}
when~$t=0.001$; the dashed line is the former and the solid line the 
latter.}
\label{fig: integrands}
\begin{center}
\includegraphics[scale=0.5]{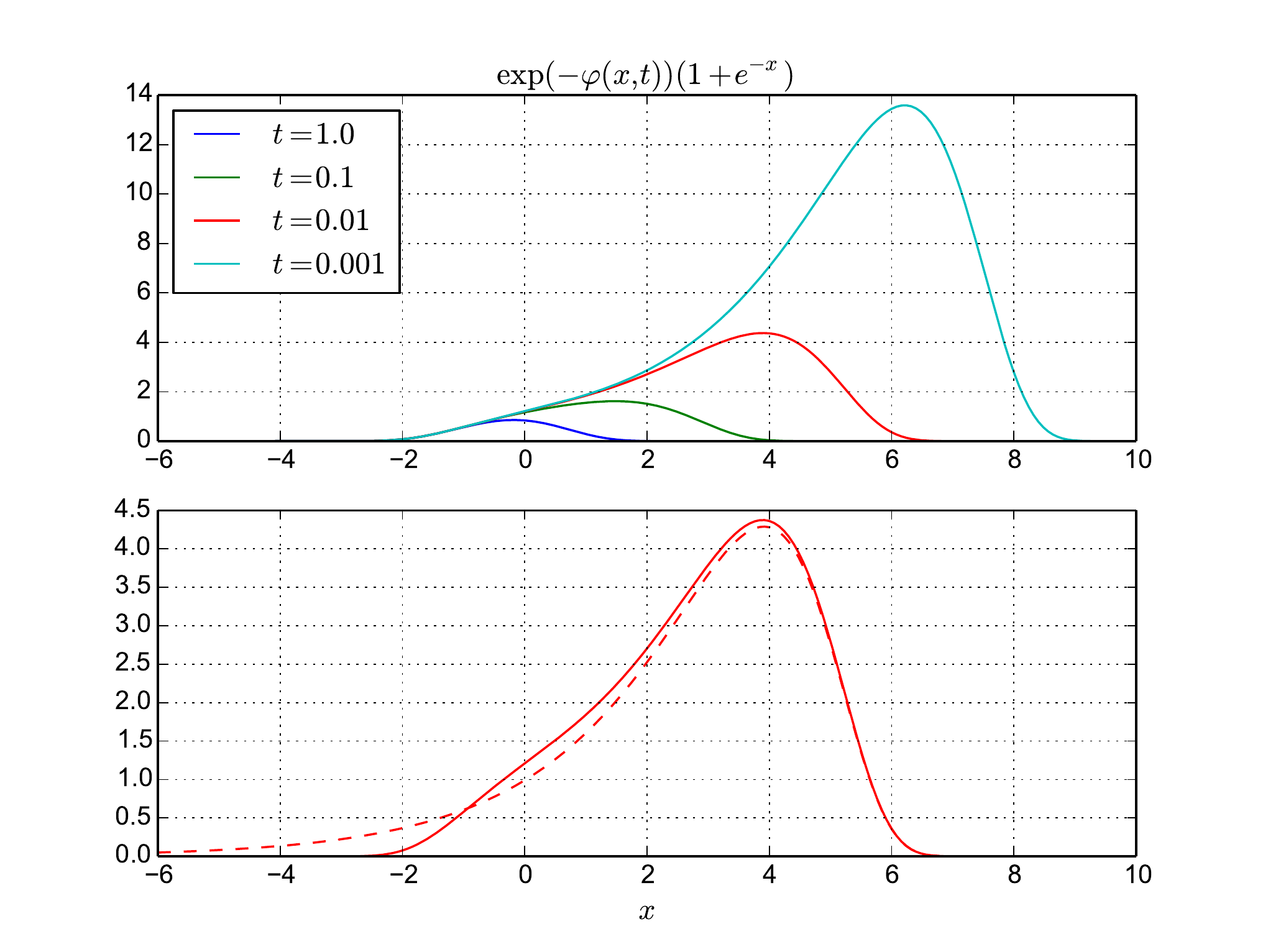} 
\end{center}
\end{figure}

Section~\ref{sec: new approx} introduces an alternative approach 
based on the substitution~$p=\exp(x-e^{-x})$, which 
transforms~\eqref{eq: basic repn} into the integral representation
\begin{equation}\label{eq: new repn}
\frac{1}{t^\beta}=\frac{1}{\Gamma(\beta)}\int_{-\infty}^\infty
	\exp\bigl(-\varphi(x,t)\bigr)(1+e^{-x})\,dx,
\end{equation}
where
\begin{equation}\label{eq: phi(x,t)}
\varphi(x,t)=tp-\beta\log p=t\exp(x-e^{-x})-\beta(x-e^{-x}).
\end{equation}
Applying the infinite trapezoidal rule again leads to an 
approximation of the form~\eqref{eq: infinite trap}, this time with
\begin{equation}\label{eq: new exponents and weights}
a_n=\exp\bigl(nh-e^{-nh}\bigr)\quad\text{and}\quad 
w_n=h(1+e^{-nh})\exp\bigl(\beta(nh-e^{-nh})\bigr).
\end{equation}
As~$x\to\infty$, the integrands in both \eqref{eq: old repn}~and 
\eqref{eq: new repn} decay like $\exp(-te^x)$.  However, they
exhibit different behaviours as~$x\to-\infty$, with the former 
decaying like~$e^{\beta x}=e^{-\beta|x|}$ whereas the latter decays 
much faster, like~$\exp(-\beta e^{-x})=\exp(-\beta e^{|x|})$, as
seen in Figure~\ref{fig: integrands} (note the differing scales
on the vertical axis).

Li~\cite{Li2010} summarised several alternative approaches for fast 
evaluation of a fractional integral of order~$\alpha$, that is, for
an integral operator of the form~\eqref{eq: Ku} with kernel
\begin{equation}\label{eq: k(t) frac int}
k(t)=\frac{t^{\alpha-1}}{\Gamma(\alpha)}=\frac{\sin\pi\alpha}{\pi}
	\int_0^\infty e^{-pt}p^{-\alpha}\,dp
	\quad\text{for $0<\alpha<1$ and $t>0$,}
\end{equation}
where the integral representation follows from~\eqref{eq: basic repn},
with~$\beta=1-\alpha$, and the reflection formula for the 
Gamma function, $\Gamma(\alpha)\Gamma(1-\alpha)=\pi/\sin\pi\alpha$.
She developed a quadrature approximation,
\begin{equation}\label{eq: Li quad}
\int_0^\infty e^{-pt}p^{-\alpha}\,dp\approx\sum_{j=1}^Q
	w_je^{-p_jt}p_j^{-\alpha}\quad\text{for $\delta\le t<\infty$,}
\end{equation}
which again provides an exponential sum approximation, and showed 
that the error can be made smaller than $\epsilon$ for 
all~$t\in[\delta,\infty)$ with~$Q$ of 
order~$(\log\epsilon^{-1}+\log\delta^{-1})^2$.  

More recently, Jiang et al.~\cite{JiangEtAl2017} developed an 
exponential sum approximation for $t\in[\delta,T]$ using 
composite Gauss quadrature on dyadic intervals, applied 
to~\eqref{eq: old repn}, with~$Q$ of order
\[
(\log\epsilon^{-1})\log\bigl(T\delta^{-1}\log\epsilon^{-1}\bigr)
	+(\log\delta^{-1})\log\bigl(\delta^{-1}\log\epsilon^{-1}\bigr).
\]
In other applications, the kernel~$k(t)$ is known via its Laplace 
transform,
\[
\hat k(z)=\int_0^\infty e^{-zt}k(t)\,dt,
\]
so that instead of the exponential sum~\eqref{eq: general k(t)} it is 
natural to seek a sum-of-poles approximation,
\[
\hat k(z)\approx\sum_{l=1}^L\frac{w_l}{z-b_l}
\]
for~$z$ in a suitable region of the complex plane; see, for instance,
Alpert, Greengard and Hagstrom~\cite{AlpertGreengardHagstrom2000}
and Xu~and Jian~\cite{XuJiang2013}.
\section{Approximation based on the substitution $p=e^x$}
\label{sec: BM approx}

The nature of the approximation~\eqref{eq: infinite trap} is revealed 
by a remarkable formula for the relative 
error~\cite[Section~2]{BeylkinMonzon2010}.  For completeness, we 
outline the proof.

\begin{theorem}\label{thm: BM rel discr error}
If the exponents and weights are given 
by~\eqref{eq: BM exponents and weights}, then the relative 
error~\eqref{eq: infinite rel err} has the representation
\begin{equation}\label{eq: rho sum}
\rho(t)=-2\sum_{n=1}^\infty R(n/h)\cos\bigl(2\pi(n/h)\log t
	-\Phi(n/h)\bigr)
\end{equation}
where $R(\xi)$~and $\Phi(\xi)$ are the real-valued functions defined 
by 
\[
\frac{\Gamma(\beta+i2\pi\xi)}{\Gamma(\beta)}=R(\xi)e^{i\Phi(\xi)}
\quad\text{with $R(\xi)>0$ and $\Phi(0)=0$.}
\]
Moreover, $R(\xi)\le e^{-2\pi\theta|\xi|}/(\cos\theta)^\beta$
for $0\le\theta<\pi/2$ and $-\infty<\xi<\infty$.
\end{theorem}
\begin{proof}
For each~$t>0$, the integrand~$f(x)=\exp(-te^x+\beta x)$ 
from~\eqref{eq: old repn} belongs to the Schwarz class of rapidly 
decreasing $C^\infty$~functions, and we may therefore apply the 
Poisson summation formula to conclude that
\[
h\sum_{n=-\infty}^\infty f(nh)=\sum_{n=-\infty}^\infty\tilde f(n/h)
	=\int_{-\infty}^\infty f(x)\,dx+\sum_{n\ne0}\tilde f(n/h),
\]
where the Fourier transform of~$f$ is 
\[
\tilde f(\xi)=\int_{-\infty}^\infty e^{-i2\pi\xi x}f(x)\,dx
	=\int_{-\infty}^\infty\exp\bigl(-te^x+(\beta-i2\pi\xi)x\bigr)\,dx.
\]
The substitution~$p=te^x$ gives
\[
\tilde f(\xi)=\frac{1}{t^{\beta-i2\pi\xi}}\int_0^\infty 
	e^{-p}p^{\beta-i2\pi\xi}\,\frac{dp}{p}
	=\frac{\Gamma(\beta-i2\pi\xi)}{t^{\beta-i2\pi\xi}},
\]
so, with $a_n$~and $w_n$ defined 
by~\eqref{eq: BM exponents and weights},
\[
\frac{1}{\Gamma(\beta)}
	\sum_{n=-\infty}^\infty w_ne^{-a_nt}
	=\frac{1}{t^\beta}+\frac{1}{t^\beta}\sum_{n\ne0}
	\frac{\Gamma(\beta-i2\pi n/h)}{\Gamma(\beta)}\,t^{i2\pi n/h}.
\]
The formula for~$\rho(t)$ follows after noting that 
$\overline{\Gamma(\beta+i2\pi\xi)}=\Gamma(\beta-i2\pi\xi)$ 
for all real~$\xi$; hence, $R(-\xi)=R(\xi)$~and 
$\Phi(-\xi)=-\Phi(\xi)$.

To estimate~$R(\xi)$, let $y>0$ and define the 
ray~$\mathcal{C}_\theta=\{\,se^{i\theta}:0<s<\infty\,\}$.  By Cauchy's 
theorem,
\[
\Gamma(\beta+iy)=\int_{\mathcal{C}_\theta}e^{-p}p^{\beta+iy}\,\frac{dp}{p}
	=\int_0^\infty e^{-se^{i\theta}}(se^{i\theta})^{\beta+iy}\,
		\frac{ds}{s}
\]
and thus
\[
|\Gamma(\beta+iy)|\le\int_0^\infty e^{-s\cos\theta}e^{-\theta y}
	s^\beta\,\frac{ds}{s}=\frac{e^{-\theta y}}{(\cos\theta)^\beta}
	\int_0^\infty e^{-s}s^\beta\,\frac{ds}{s}
	=\frac{e^{-\theta y}}{(\cos\theta)^\beta}\,\Gamma(\beta),
\]
implying the desired bound for~$R(\xi)$.
\end{proof}

In practice, the amplitudes~$R(n/h)$ decay so rapidly with~$n$ that 
only the first term in the expansion~\eqref{eq: rho sum} is 
significant.  For instance, since~\cite[6.1.30]{AbramowitzStegun1965}
\[
\bigl|\Gamma(\tfrac12+iy)\bigr|^2=\frac{\pi}{\cosh(\pi y)},
\]
if $\beta=1/2$ then 
$R(\xi)=(\cosh2\pi^2\xi)^{-1/2}\le\sqrt{2}e^{-\pi^2\xi}$ so,
choosing~$h=1/3$, we have $R(1/h)=1.95692\times10^{-13}$~and
$R(2/h)=2.70786\times10^{-26}$.  In general,
the bound~$R(n/h)\le e^{-2\pi\theta n/h}/(\cos\theta)^\beta$ from 
the theorem is minimized by choosing $\tan\theta=2\pi n/(\beta h)$,
implying that
\[
R(n/h)\le\bigl(1+(r_n/\beta)^2\bigr)^{\beta/2}
	\exp\bigl(-r_n\arctan(r_n/\beta)\bigr)
	\quad\text{where $r_n=2\pi n/h$.}
\]

Since we can evaluate only a \emph{finite} exponential sum, we now 
estimate the two tails of the infinite sum in terms of the upper 
incomplete Gamma function,
\[
\Gamma(\beta,q)=\int_q^\infty e^{-p}p^\beta\,\frac{dp}{p}
	\quad\text{for $\beta>0$~and $q>0$.}
\]

\begin{theorem}\label{thm: BM rel trunc error}
If the exponents and weights are given 
by~\eqref{eq: BM exponents and weights}, then
\[
t^\beta\sum_{n=N+1}^\infty w_ne^{-a_nt}\le\Gamma(\beta,te^{Nh})
	\quad\text{provided $te^{Nh}\ge\beta$,}
\]
and
\[
t^\beta\sum_{n=-\infty}^{-M-1}w_ne^{-a_nt}\le\Gamma(\beta)
	-\Gamma(\beta,te^{-Mh})
	\quad\text{provided $te^{-Mh}\le\beta$.}
\]
\end{theorem}
\begin{proof}
For each~$t>0$, the integrand~$f(x)=\exp(-te^x+\beta x)$ 
from~\eqref{eq: old repn} decreases for~$x>\log(\beta/t)$.  Therefore,
if $Nh\ge\log(\beta/t)$, that is, if $te^{Nh}\ge\beta$, then
\[
t^\beta h\sum_{n=N+1}^\infty f(nh)\le t^\beta\int_{Nh}^\infty f(x)\,dx
	=\int_{te^{Nh}}^\infty e^{-p}p^\beta\,\frac{dp}{p}
		=\Gamma(\beta,te^{Nh}),
\]
where, in the final step, we used the substitution~$p=te^x$.  
Similarly, the function~$f(-x)=\exp(-te^{-x}-\beta x)$ decreases 
for~$x>\log(t/\beta)$ so if $Mh\ge\log(t/\beta)$, that is, if
$te^{-Mh}\le\beta$, then
\[
t^\beta h\sum_{n=-\infty}^{-M-1}f(nh)
	=t^\beta h\sum_{n=M+1}^\infty f(-nh)
	\le t^\beta\int_{Mh}^\infty f(-x)\,dx
	=\int_0^{te^{-Mh}}e^{-p}p^\beta\,\frac{dp}{p},
\]
where, in the final step, we used the substitution~$p=te^{-x}$.
\end{proof}

Given $\epsilon_{\text{RD}}>0$ there exists~$h>0$ such that
\begin{equation}\label{eq: rel discr error}
2\sum_{n=1}^\infty|\Gamma(\beta+i2\pi n/h)|
	=\epsilon_{\text{RD}}\Gamma(\beta),
\end{equation}
and by Theorem~\ref{thm: BM rel discr error},
\[
|\rho(t)|\le\epsilon_{\text{RD}}\quad\text{for $0<t<\infty$,}
\]
so $\epsilon_{\text{RD}}$ is an upper bound for the \emph{relative 
discretization} error.  Similarly, given a sufficiently 
small~$\epsilon_{\text{RT}}>0$, there exist $x_\delta>0$~and 
$X_T>0$ such that $\delta e^{x_\delta}\ge\beta$~and 
$Te^{-X_T}\le\beta$ with
\begin{equation}\label{eq: rel trunc error}
\Gamma(\beta,\delta e^{x_\delta})=\epsilon_{\text{RT}}\Gamma(\beta)
\quad\text{and}\quad
\Gamma(\beta)-\Gamma(\beta,Te^{-X_T})
	=\epsilon_{\text{RT}}\Gamma(\beta).
\end{equation}
Thus, by Theorem~\ref{thm: BM rel trunc error}, 
\[
\frac{t^\beta}{\Gamma(\beta)}\sum_{n=N+1}^\infty w_ne^{-a_nt}
	\le\epsilon_{\text{RT}}
	\quad\text{for $t\ge\delta$~and $Nh\ge x_\delta$,}
\]
and
\[
\frac{t^\beta}{\Gamma(\beta)}\sum_{n=-\infty}^{-M-1}w_ne^{-a_nt}
	\le\epsilon_{\text{RT}}
	\quad\text{for $t\le T$ and $Mh\ge X_T$,}
\]
showing that $2\epsilon_{\text{RT}}$ is an upper bound for the 
\emph{relative truncation} error.  Denoting the overall
relative error for the finite sum~\eqref{eq: finite trap} by
\begin{equation}\label{eq: rho M N}
\rho^N_M(t)=1-\frac{t^\beta}{\Gamma(\beta)}\sum_{n=-M}^Nw_ne^{-a_nt},
\end{equation}
we therefore have
\begin{equation}\label{eq: overall rel error}
|\rho^N_M(t)|\le\epsilon_{\text{RD}}+2\epsilon_{\text{RT}}
\quad\text{for $\delta\le t\le T$, $Nh\ge x_\delta$ and $Mh\ge X_T$.}
\end{equation}

\begin{figure}
\caption{The bound~$\epsilon_{\text{RD}}$ for the
relative discretization error, defined by~\eqref{eq: rel discr error},
as a function of~$1/h$ for various choices of~$\beta$.}
\label{fig: rel discr error}
\begin{center}
\includegraphics[scale=0.5]{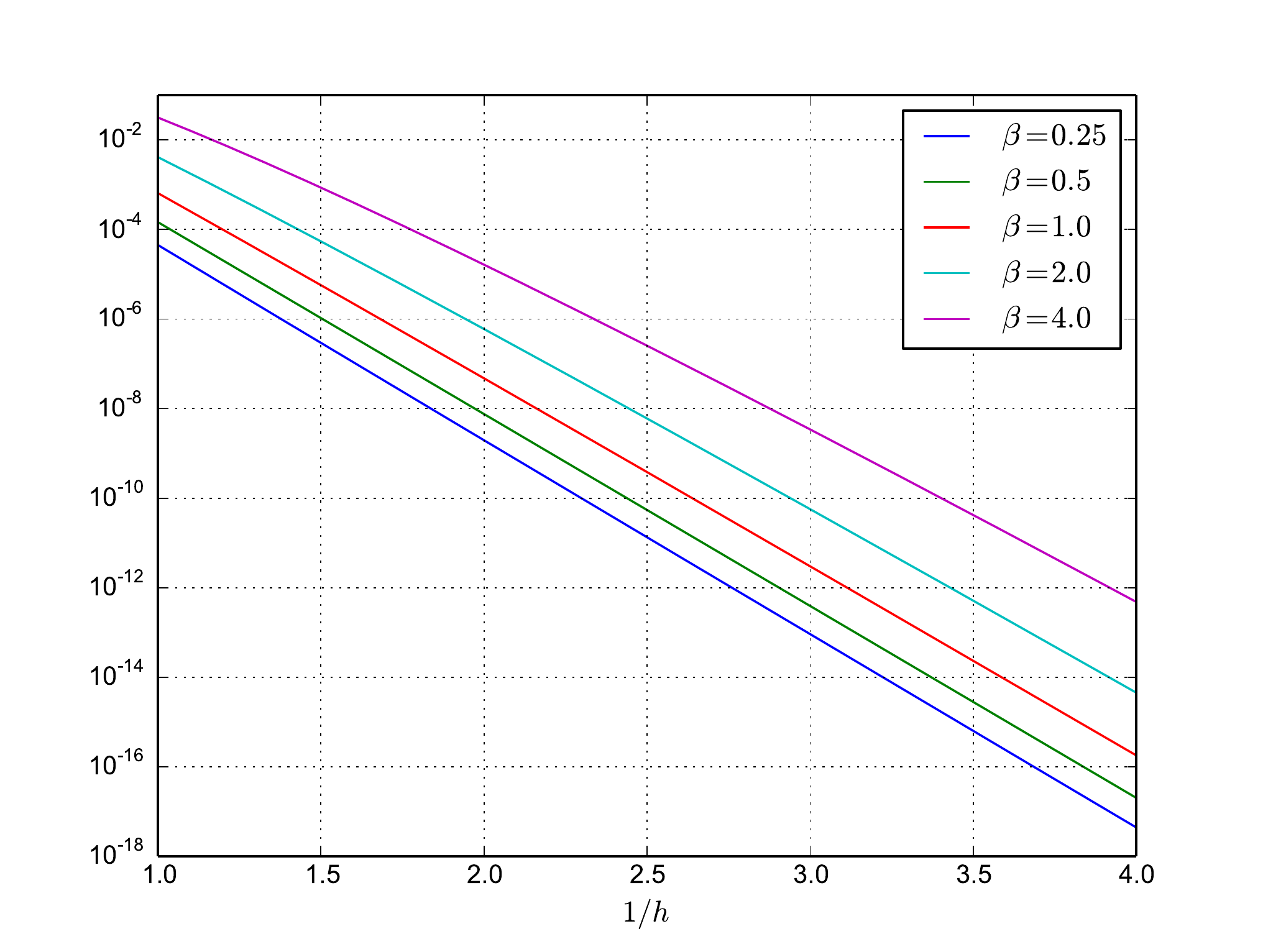}
\end{center}
\end{figure}

\begin{figure}
\caption{The growth in $M$~and $N$ as the upper bound for the overall 
relative error~\eqref{eq: rho M N} decreases, for different choices of 
$T$~and $\delta$.}
\label{fig: M N}
\begin{center}
\includegraphics[scale=0.5]{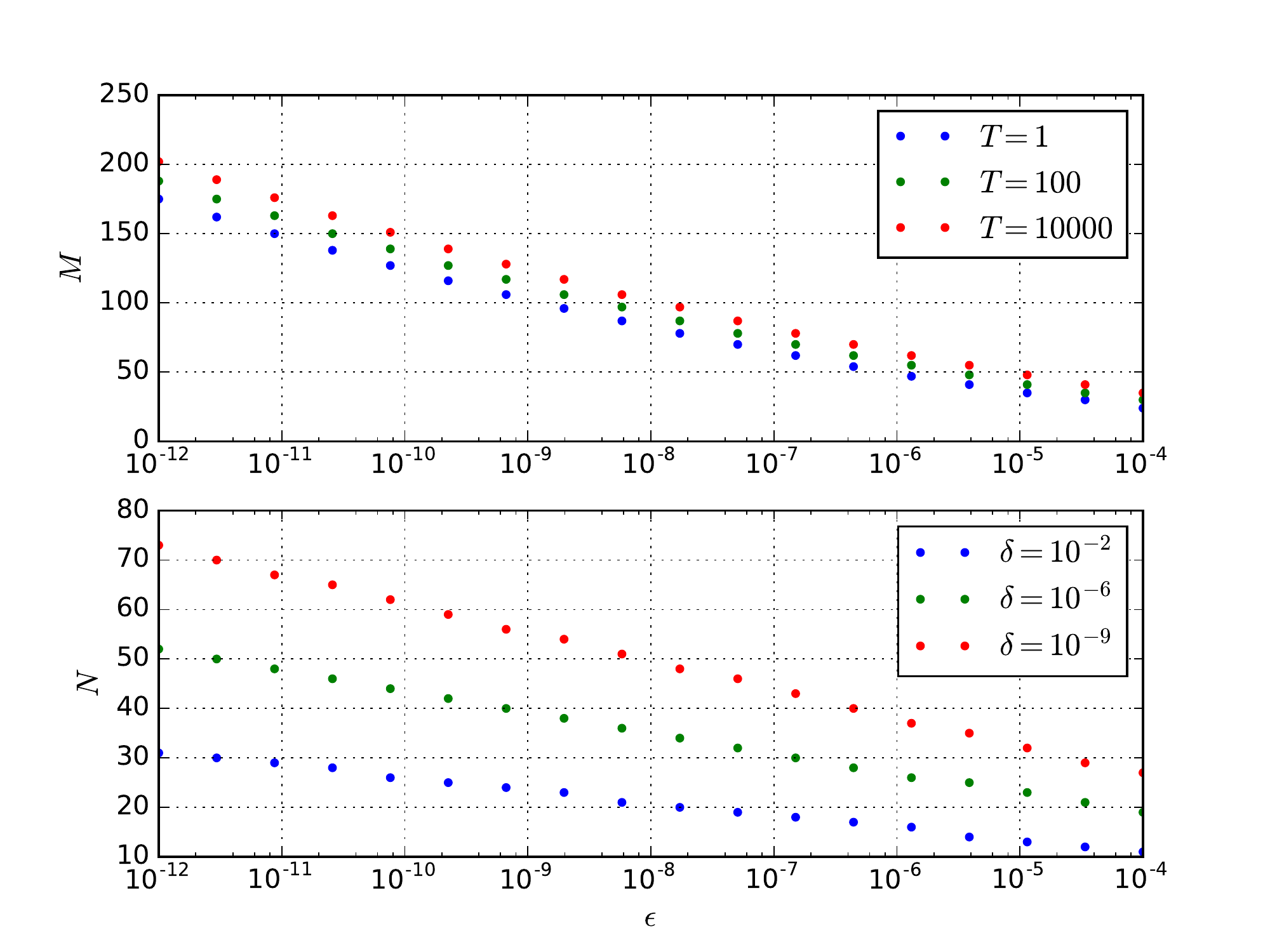}
\end{center}
\end{figure}

The estimate for~$R(\xi)$ in Theorem~\ref{thm: BM rel discr error},
together with the asymptotic behaviours
\[
\Gamma(\beta,q)\sim q^{\beta-1}e^{-q}\quad\text{as~$q\to\infty$},
\]
and
\[
\Gamma(\beta)-\Gamma(\beta,q)\sim\frac{q^\beta}{\beta}
	\quad\text{as $q\to0$,}
\]
imply that \eqref{eq: overall rel error} can be satisfied with
\[
h^{-1}\ge C\log\epsilon_{\text{RD}}^{-1},\qquad
N\ge Ch^{-1}\log(\delta^{-1}\log\epsilon_{\text{RT}}^{-1}),\qquad
M\ge Ch^{-1}\log(T\epsilon_{\text{RT}}^{-1}).
\]
Figure~\ref{fig: rel discr error} shows the relation between 
$\epsilon_{\text{RD}}$~and $1/h$ given by~\eqref{eq: rel discr error},
and confirms that $1/h$ is approximately proportional 
to~$\log\epsilon_{\text{RT}}^{-1}$.  In Figure~\ref{fig: M N}, for 
each value of~$\epsilon$ we computed $h$ by 
solving~\eqref{eq: rel discr error} 
with~$\epsilon_{\text{RD}}=\epsilon/3$, then computed 
$x_\delta$~and $X_T$ by solving~\eqref{eq: rel trunc error} 
with~$\epsilon_{\text{RT}}=\epsilon/3$, and 
finally put $M=\lceil X_T/h\rceil$~and $N=\lceil x_\delta/h\rceil$.

\section{Prony's method}\label{sec: Prony}
The construction of Section~\ref{sec: BM approx} leads to an
exponential sum approximation~\eqref{eq: finite trap} with many 
small exponents~$a_n$.  We will now explain how the 
corresponding terms can be aggregated to yield a more efficient 
approximation.

Consider more generally an exponential sum
\[
g(t)=\sum_{l=1}^L w_le^{-a_lt},
\]
in which the weights and exponents are all strictly positive. Our aim 
is to approximate this function by an exponential sum with fewer 
terms,
\[
g(t)\approx\sum_{k=1}^K\tilde w_ke^{-\tilde a_kt},\quad 2K-1<L,
\]
whose weights~$\tilde w_k$ and exponents~$\tilde a_k$ are again all 
strictly positive.  To this end, let
\[
g_j=(-1)^jg^{(j)}(0)=\sum_{l=1}^L w_la_l^j.
\]
We can hope to find $2K$~parameters $\tilde w_k$~and $\tilde a_k$ 
satisfying the $2K$~conditions
\begin{equation}\label{eq: g tilde w a}
g_j=\sum_{k=1}^K\tilde w_k\tilde a_k^j
	\quad\text{for $0\le j\le 2K-1$,} 
\end{equation}
so that, by Taylor expansion,
\[
g(t)\approx\sum_{j=0}^{2K-1}g_j\,\frac{(-t)^j}{j!}
	=\sum_{k=1}^K\tilde w_k\sum_{j=0}^{2K-1}
	\frac{(-\tilde a_kt)^j}{j!}
	\approx\sum_{k=1}^K\tilde w_ke^{-\tilde a_kt}.
\]
The approximations here require that the $g_j$~and the $\tilde a_kt$
are nicely bounded, and preferably small.

In Prony's method, we seek to satisfy~\eqref{eq: g tilde w a} by 
introducing the monic polynomial
\[
Q(z)=\prod_{k=1}^K(z-\tilde a_k)=\sum_{k=0}^K q_kz^k,
\]
and observing that the unknown coefficients~$q_k$ must satisfy
\[
\sum_{m=0}^K g_{j+m}q_m=\sum_{m=0}^K\sum_{k=1}^K\tilde w_k
	\tilde a_k^{j+m}q_m=\sum_{k=1}^K\tilde w_k\tilde a_k^j
	\sum_{m=0}^Kq_m\tilde a_k^m
	=\sum_{k=1}^K\tilde w_k\tilde a_k^jQ(\tilde a_k)=0,
\]
for~$0\le j\le K-1$ (so that $j+m\le2K-1$ for $1\le m\le K$), 
with~$q_K=1$.  Thus,
\[
\sum_{m=0}^{K-1}g_{j+m}q_m=b_j,\quad\text{where $b_j=-g_{j+K}$,}
\quad\text{for $0\le j\le K-1$,}
\]
which suggests the procedure \textit{Prony} defined in 
Algorithm~\ref{alg: Prony}.  We must, however, beware of several 
potential pitfalls:
\begin{enumerate}
\item the best choice for~$K$ is not clear;
\item the $K\times K$~matrix $[g_{j+k}]$ might be badly conditioned;
\item the roots of the polynomial~$Q(z)$ might not all be real and
positive;
\item the linear system for the $\tilde w_k$ is overdetermined, and 
the least-squares solution might have large residuals;
\item the $\tilde w_k$ might not all be positive.
\end{enumerate}
We will see that nevertheless the algorithm can be quite effective, 
even when~$K=1$, in which case we simply compute
\[
g_0=\sum_{l=1}^Lw_l,\quad
g_1=\sum_{l=1}^Lw_la_l,\quad
\tilde a_1=g_1/g_0,\quad\tilde w_1=g_0.
\]

\begin{algorithm}
\caption{$\textit{Prony}(a_1,\dots,a_L,w_1,\dots w_L, K)$}
\label{alg: Prony}
\begin{algorithmic}
\REQUIRE $2K-1\le L$
\STATE Compute $g_j=\sum_{l=1}^Lw_la_l^j$ for $0\le j\le2K-1$
\STATE Find $q_0$, \dots, $q_{K-1}$ satisfying
$\sum_{m=0}^{K-1}g_{j+m}q_m=-g_{j+K}$ for $0\le j\le K-1$,
and put $q_K=1$
\STATE Find the roots $\tilde a_1$, \dots, $\tilde a_K$ of the 
polynomial $Q(z)=\sum_{k=0}^K q_kz^k$
\STATE Find $\tilde w_1$, \dots, $\tilde w_K$ satisfying
$\sum_{k=1}^K\tilde a_k^j\tilde w_k\approx g_j$ for~$0\le j\le2K-1$
\RETURN $\tilde a_1$, \dots, $\tilde a_K$, $\tilde w_1$, \dots,
$\tilde w_K$
\end{algorithmic}
\end{algorithm}

\begin{table}
\caption{Performance of Prony's method for different $L$~and $K$ 
using the parameters of Example~\ref{ex: Prony}.  For each~$K$,
we seek the largest~$L$ for which the maximum relative error is
less than $\epsilon=10^{-8}$.} \label{tab: Prony}
\begin{center}
\begin{tabular}{c|cccccc}
$\quad L\quad$&$\quad K=1\quad$&$\quad K=2\quad$&$\quad K=3\quad$
&$\quad K=4\quad$&$\quad K=5\quad$&$\quad K=6\quad$\\ 
\hline
 66& 9.64e-01& 4.30e-01& 6.15e-02& 3.02e-03& 4.77e-05& 2.29e-07\\
 65& 8.11e-01& 1.69e-01& 9.89e-03& 1.80e-04& 9.98e-07&
\textbf{1.66e-09}\\
 64& 5.35e-01& 4.59e-02& 1.03e-03& 6.85e-06& 1.35e-08& 7.96e-12\\
 63& 2.72e-01& 9.17e-03& 7.76e-05& 1.89e-07&\textbf{1.36e-10}& 
2.74e-14\\
 62& 1.12e-01& 1.46e-03& 4.64e-06&\textbf{4.19e-09}& 1.11e-12& 
3.58e-16\\
 61& 3.99e-02& 1.98e-04& 2.38e-07& 8.05e-11& 8.28e-15& 3.52e-16\\
 60& 1.28e-02& 2.43e-05& 1.10e-08& 1.41e-12& 4.63e-16& 2.24e-16\\
 59& 3.82e-03& 2.78e-06&\textbf{4.81e-10}& 2.36e-14& 4.63e-16& 
1.25e-16\\
 58& 1.10e-03& 3.05e-07& 2.02e-11& 4.46e-16& 1.23e-16& 6.27e-17\\
 57& 3.07e-04& 3.27e-08& 8.25e-13& 5.60e-17& 8.40e-17&         \\
 56& 8.43e-05&\textbf{3.44e-09}& 3.32e-14& 8.96e-17& 5.60e-17&        
 \\
 55& 2.29e-05& 3.59e-10& 1.32e-15& 4.48e-17& 4.48e-17&         \\
   &         &         &         &         &         &         \\
 48&\textbf{2.30e-09}& 3.98e-17& 2.58e-18&         &         &        
 \\
 47& 6.16e-10& 3.92e-18& 1.54e-18&         &         &         \\
\end{tabular}
\end{center}
\end{table}

\begin{figure}
\caption{Top: the additional contribution~$|\eta(t)|$ to the relative 
error from applying Prony's method in Example~\ref{ex: Prony} with 
$L=65$~and $K=6$. Bottom: the overall relative error for the 
resulting approximation~\eqref{eq: reduced approx} of~$t^{-\beta}$ 
requiring $L-K=59$~fewer terms.}
\label{fig: Prony}
\begin{center}
\includegraphics[scale=0.5]{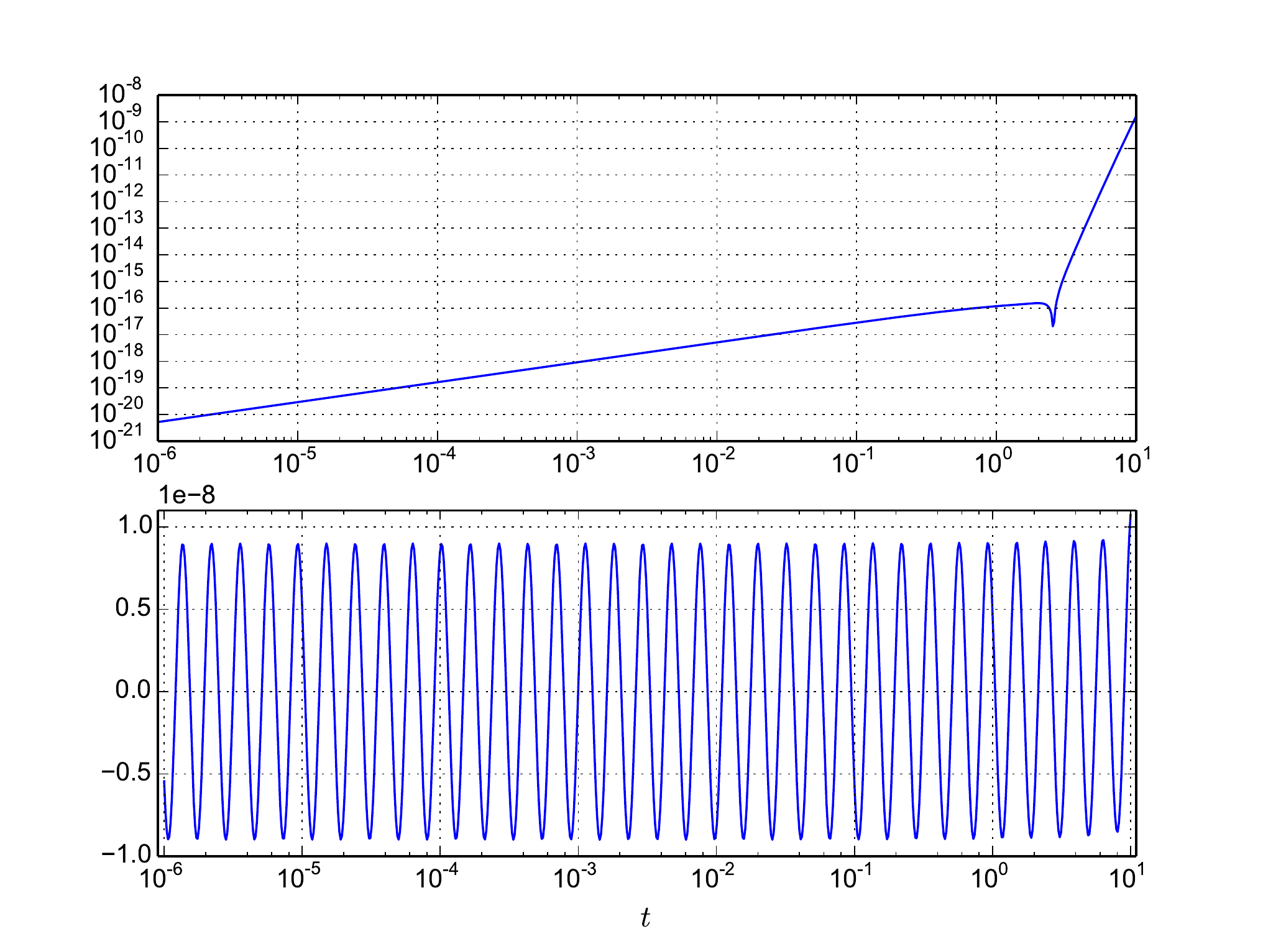}
\end{center}
\end{figure}

\begin{example}\label{ex: Prony}
We took $\beta=3/4$, $\delta=10^{-6}$, $T=10$, $\epsilon=10^{-8}$, 
$\epsilon_{\text{RD}}=0.9\times10^{-8}$~and 
$\epsilon_{RT}=0.05\times10^{-8}$. The methodology of 
Section~\ref{sec: BM approx} led to the choices $h=0.47962$,
$M=65$~and $N=36$, and we confirmed via direct evaluation of the 
relative error that $|\rho^N_M(t)|\le0.92\times10^{-8}$
for~$\delta\le t\le T$.  We applied Prony's method to the first 
$L$~terms of the sum in~\eqref{eq: finite trap}, that is, those 
with~$-M\le n\le L-M$, thereby reducing the total number of terms 
by~$L-K$.   Table~\ref{tab: Prony} lists, for different choices of 
$L$~and $K$, the additional contribution to the relative error, that 
is, $\max_{1\le p\le P}|\eta(t_p)|$ where 
\begin{equation}
\eta(t)=\frac{t^\beta}{\Gamma(\beta)}\biggl( 
	\sum_{k=1}^K\tilde w_ke^{-\tilde a_kt}
	-\sum_{l=1}^Lw_{l'}e^{-a_{l'}t}\biggr),\qquad l'=l-M+1,
\end{equation}
and we use a geometric grid in~$[\delta,1]$ 
given by~$t_p=T^{(p-1)/(P-1)}\delta^{(P-p)/(P-1)}$ for~$1\le p\le P$ 
with~$P=751$.
The largest reduction consistent with maintaining overall accuracy was 
when $L=65$~and $K=6$, and Figure~\ref{fig: Prony} (Top) plots
$|\eta(t)|$ in this case, as well as the overall relative 
error (Bottom) for the resulting approximation,
\begin{equation}\label{eq: reduced approx}
\frac{1}{t^\beta}\approx\frac{1}{\Gamma(\beta)}\biggl(
	\sum_{k=1}^K\tilde w_ke^{-\tilde a_kt}+
	\sum_{n=L-M}^N w_ne^{-a_nt}\biggr)
		\quad\text{for $10^{-6}\le t\le10$.}
\end{equation}
In this way, the number of terms in the exponential sum 
approximation was reduced from~$M+1+N=102$ to~$(M+K-L)+1+N=43$, with 
the maximum absolute value of the relative error growing only slightly 
to~$1.07\times10^{-8}$.  Figure~\ref{fig: Prony} (Bottom) shows that
the relative error is closely approximated by the first term 
in~\eqref{eq: rho sum}, that is, 
$\rho^M_N(t)\approx-2R(h^{-1})\cos\bigl(2\pi h^{-1}\log t-\Phi(h^{-1})
\bigr)$ for~$\delta\le t\le T$.
\end{example}

\section{Approximation based on the substitution $p=\exp(x-e^{-x})$}
\label{sec: new approx}
We now consider the alternative exponents and weights given 
by~\eqref{eq: new exponents and weights}.  A different approach is 
needed for the error analysis, and we define
\[
\mathcal{I}(f)=\int_{-\infty}^\infty f(x)\,dx
\quad\text{and}\quad
\mathcal{Q}(f,h)=h\sum_{n=-\infty}^\infty f(nh)\quad\text{for $h>0$,}
\]
so that $\mathcal{Q}(f,h)$ is an infinite trapezoidal rule 
approximation to~$\mathcal{I}(f)$.  Recall the following well-known
error bound.

\begin{theorem}\label{thm: trapezoidal error}
Let $r>0$. Suppose that $f(z)$ is continuous on the closed 
strip~$|\Im z|\le r$, analytic on the open strip~$|\Im z|<r$, and 
satisfies
\[
\int_{-\infty}^\infty\bigl(|f(x+ir)|+|f(x-ir)|\bigr)\,dx\le A_r
\]
with
\[
\int_{-r}^r|f(x\pm iy)|\,dy\to0\quad\text{as $|x|\to\infty$.}
\]
Then, for all $h>0$,
\[
|\mathcal{Q}(f,h)-\mathcal{I}(f)|
	\le\frac{A_re^{-2\pi r/h}}{1-e^{-2\pi r/h}}.
\]
\end{theorem}
\begin{proof}
See McNamee, Stenger and  Whitney~\cite[Theorem~5.2]{McNameeEtAl1971}.
\end{proof}

For~$t>0$, we define the entire analytic function of~$z$,
\begin{equation}\label{eq: f(z) def}
f(z)=\exp\bigl(-\varphi(z,t)\bigr)(1+e^{-z}),
\end{equation}
where $\varphi(z,t)$ is the analytic continuation of the function 
defined in~\eqref{eq: phi(x,t)}.
In this way, $t^{-\beta}=\mathcal{I}(f)/\Gamma(\beta)$ 
by~\eqref{eq: new repn}.

\begin{lemma}\label{lem: M}
If $0<r<\pi/2$, then the function~$f$ defined 
in~\eqref{eq: f(z) def} satisfies the hypotheses of 
Theorem~\ref{thm: trapezoidal error} with~$A_r\le Ct^{-\beta}$ for 
$0<t\le1$, where the constant~$C>0$ depends only on $\beta$~and $r$.
\end{lemma}
\begin{proof}
A short calculation shows that
\[
\Re\varphi(x\pm iy,t)=t\exp(x-e^{-x}\cos y)\cos(y+e^{-x}\sin y)
	-\beta(x-e^{-x}\cos y),
\]
and that if $0<\epsilon<\pi/2-r$, then
\begin{equation}\label{eq: x*}
0\le y+e^{-x}\sin y\le\frac{\pi}{2}-\epsilon
	\quad\text{for $x\ge x^*=\log\frac{\sin r}{\pi/2-r-\epsilon}$
		and $0\le y\le r$.}
\end{equation}
Thus, if $x\ge x^*$ then
$\cos(r+e^{-x}\sin r)\ge\cos(\pi/2-\epsilon)=\sin\epsilon$ so
\[
\Re\varphi(x\pm ir,t)\ge t\exp(x-e^{-x^*}\cos r)\sin\epsilon
	-\beta x+\beta e^{-x}\cos r\ge cte^x-\beta x,
\]
where $c=\exp(-e^{-x^*}\cos r)\sin\epsilon>0$. If necessary, we 
increase $x^*$ so that $x^*>0$.  Since 
$|1+e^{-(x\pm ir)}|\le1+e^{-x}$,
\begin{align*}
\int_{x^*}^\infty|f(x\pm ir)|\,dx
	&=\int_{x^*}^\infty\exp\bigl(-\Re\varphi(x\pm ir,t)\bigr)
	\bigl|1+e^{-(x\pm ir)}\bigr|\,dx\\
	&\le\int_{x^*}^\infty\exp(-cte^x+\beta x)(1+e^{-x})\,dx,
\end{align*}
and the substitution~$p=e^x$ then yields, with $p^*=e^{x^*}$,
\begin{multline*}
\int_{x^*}^\infty|f(x\pm ir)|\,dx
	\le\int_{p^*}^\infty e^{-ctp}p^\beta(1+p^{-1})\,\frac{dp}{p}
	\le\bigl(1+(p^*)^{-1}\bigr)\int_{p^*}^\infty 
		e^{-ctp}p^\beta\,\frac{dp}{p}\\
	=\frac{1+(p^*)^{-1}}{(ct)^\beta}\int_{ctp^*}^\infty e^{-p}p^\beta
		\,\frac{dp}{p}
	\le\frac{1+(p^*)^{-1}}{(ct)^\beta}\int_0^\infty 
		e^{-p}p^\beta \,\frac{dp}{p}\equiv Ct^{-\beta}.
\end{multline*}
Also, if $x\ge0$ then
\[
\Re\varphi(x\pm ir,t)\ge-t\exp(x-e^{-x}\cos r)-\beta(x-e^{-x}\cos r)
	\ge-te^x-\beta x
\]
so
\[
\int_0^{x^*}|f(x\pm ir)|\,dx
	\le\int_0^{x^*}\exp(te^x+\beta x)(1+e^{-x})\,dx
	\le2x^*\exp\bigl(te^{x^*}+\beta x^*\bigr),
\]
which is bounded for~$0<t\le1$.
Similarly, if $x\le0$ then $\exp(x-e^{-x}\cos r)\le1$ so
$\Re\varphi(x\pm ir,t)\ge-t+\beta e^{-x}\cos r$ and therefore,
using again the substitution~$p=e^x$,
\begin{align*}
\int_{-\infty}^0|f(x\pm ir)|\,dx
	&\le\int_{-\infty}^0\exp(t-\beta e^{-x}\cos r)(1+e^{-x})\,dx\\
	&=\int_0^\infty\exp(t-\beta e^x\cos r)(1+e^x)\,dx
	=e^t\int_1^\infty e^{-\beta p\cos r}(1+p)\,\frac{dp}{p},
\end{align*}
which is also bounded for~$0<t\le1$.  The required estimate for~$A_r$ 
follows.

If $x\ge x^*$, then the preceding inequalities based 
on~\eqref{eq: x*} show that 
\[
\int_{-r}^r|f(x+iy)|\,dy
	\le2r\max_{|y|\le r}|f(x+iy)|
	\le2r\exp(-cte^x+\beta x)(1+e^{-x}),
\]
which tends to zero as~$x\to\infty$ for any~$t>0$.  Similarly,
if $x\le0$, then $\Re\varphi(x\pm iy)\ge-t+\beta e^{-x}\cos r$
for~$|y|\le r$, so
\[
\int_{-r}^r|f(x+iy)|\,dy\le2r\exp(t-\beta e^{-x}\cos r)(1+e^{-x}),
\]
which again tends to zero as~$x\to-\infty$.
\end{proof}

Together, Theorem~\ref{thm: trapezoidal error} and Lemma~\ref{lem: M}
imply the following bound for the relative 
error~\eqref{eq: infinite rel err} in the infinite exponential sum 
approximation~\eqref{eq: infinite trap}.

\begin{theorem}\label{thm: Q(f) error}
Let $h>0$ and define $a_n$~and $w_n$ 
by~\eqref{eq: new exponents and weights}. If $0<r<\pi/2$, then 
there exists a constant~$C_1$ (depending on $\beta$~and $r$) such that
\[
|\rho(t)|\le C_1 e^{-2\pi r/h}\quad\text{for $0<t\le1$.}
\]
\end{theorem}
\begin{proof}
The definitions above mean that $hf(nh)=w_ne^{-a_nt}$.
\end{proof}

Thus, a relative accuracy~$\epsilon$ is achieved by choosing~$h$ 
of order~$1/\log\epsilon^{-1}$.  Of course, in practice we must 
compute a finite sum, and the next lemma estimates the two parts of 
the associated truncation error.

\begin{lemma}\label{lem: trunc}
Let $h>0$, $0<\theta<1$ and $0<t\le1$.  Then the 
function~$f$ defined in~\eqref{eq: f(z) def} satisfies
\begin{equation}\label{eq: Mh}
\frac{h}{\Gamma(\beta)}\sum_{M=-\infty}^{-M-1}f(nh)
	\le C_2\exp(-\beta e^{Mh})\quad\text{for}\quad
	Mh\ge\begin{cases}\log(\beta^{-1}-1),&0<\beta<1/2,\\
      0,&\beta\ge1/2,
\end{cases}
\end{equation}
and
\begin{equation}\label{eq: Nh}
\frac{h}{\Gamma(\beta)}\sum_{n=N+1}^\infty f(nh)
	\le\frac{C_3}{t^\beta}\,\exp\bigl(-\theta te^{Nh-1}\bigr)
\quad\text{for}\quad Nh\ge1+\log(\beta t^{-1}).
\end{equation}
When $0<\beta\le1$, the second estimate holds also with~$\theta=1$.
\end{lemma}
\begin{proof}
If $n\le0$, then $\varphi(nh,t)\ge-t+\beta e^{-nh}$ so
\[
f(nh)\le g_1(-nh)\quad\text{where}\quad 
	g_1(x)=\exp(t-\beta e^x)(1+e^x).
\]
The function~$g_1(x)$ decreases for~$x>\log(\beta^{-1}-1)$ if 
$0<\beta<1/2$, and for all $x\ge0$ if $\beta\ge1/2$, so
\[
h\sum_{n=-\infty}^{-M-1}f(nh)\le h\sum_{n=M+1}^\infty g_1(nh)
	\le\int_{Mh}^\infty g_1(x)\,dx
	\quad\text{for $M$ as in~\eqref{eq: Mh},}
\]
and the substitution~$p=e^x$ gives
\[
\int_{Mh}^\infty g_1(x)\,dx=\int_{e^{Mh}}^\infty e^{t-\beta p}
	(1+p)\,\frac{dp}{p}
	\le2e^t\int_{e^{Mh}}^\infty e^{-\beta p}\,dp
	=\frac{2e^t}{\beta}\exp(-\beta e^{Mh}),
\]
so the first estimate holds with $C_2=2e/\Gamma(\beta+1)$.

If $n\ge0$ we have $\varphi(nh,t)\ge t\exp(nh-1)-\beta nh$~and
$1+e^{-nh}\le2$, so 
\[
f(nh)\le g_2(nh)\quad\text{where}\quad 
	g_2(x)=2\exp(-te^{x-1}+\beta x).
\]
The function~$g_2(x)$ decreases for~$x>1+\log(\beta t^{-1})$, so
\[
h\sum_{n=N+1}^\infty f(nh)\le\int_{Nh}^\infty g_2(x)\,dx
	\quad\text{for $N$ as in \eqref{eq: Nh},}
\]
and the substitution~$p=e^x$ gives
\[
\int_{Nh}^\infty g_2(x)\,dx
	\le 2\int_{e^{Nh}}^\infty e^{-te^{-1}p}p^\beta\,\frac{dp}{p}\\
	=2\biggl(\frac{e}{t}\biggr)^\beta
	\int_{te^{Nh-1}}^\infty e^{-p}p^{\beta-1}\,dp.
\]
Since $te^{Nh-1}\ge\beta$, if $0<\beta\le1$ then the integral on the 
right is bounded above by~$\beta^{\beta-1}\exp(-te^{Nh-1})$.  If 
$\beta>1$, then $p^{\beta-1}e^{-(1-\theta)p}$ is bounded for~$p>0$ so
\[
\int_{te^{Nh-1}}^\infty e^{-p}p^{\beta-1}\,dp
=\int_{te^{Nh-1}}^\infty e^{-\theta p}
	(p^{\beta-1}e^{-(1-\theta)p})\,dp\le C\exp(-\theta te^{Nh-1}),
\]
completing the proof.
\end{proof}

It is now a simple matter to see that the number of terms~$L=M+1+N$ 
needed to ensure a relative accuracy~$\epsilon$ 
for~$\delta\le t\le1$ is of 
order~$(\log\epsilon^{-1})\log(\delta^{-1}\log\epsilon^{-1})$.

\begin{theorem}\label{thm: h M N}
Let $a_n$~and $w_n$ be defined 
by~\eqref{eq: new exponents and weights}.
For $0<\delta\le1$ and for a sufficiently small~$\epsilon>0$, if
\[
\frac{1}{h}\ge\frac{1}{2\pi r}\,\log\frac{3C_1}{\epsilon},
\quad
M\ge\frac{1}{h}\,\log\biggl(\frac{1}{\beta}\,
	\log\frac{3C_2}{\epsilon}\biggr),\quad
N\ge1+\frac{1}{h}\,\log\biggl(\frac{1}{\theta\delta}\,
	\log\frac{3C_3}{\epsilon}\biggr),
\]
then 
\[
|\rho^N_M(t)|\le\epsilon\quad\text{for $\delta\le t\le1$.}
\]
\end{theorem}
\begin{proof}
The inequalities for $h$, $M$~and $N$ imply that each of
$C_1e^{-2\pi r/h}$, $C_2\exp(-\beta e^{Mh})$ and 
$C_3\exp(-\theta te^{Nh-1})$ is bounded 
above by~$\epsilon t^{-\beta}/3$, so the error estimate is a 
consequence of Theorem~\ref{thm: Q(f) error}, Lemma~\ref{lem: trunc} 
and the triangle inequality.  Note that the restrictions on $M$~and 
$N$ in \eqref{eq: Mh}~and \eqref{eq: Nh} will be satisfied 
for~$\epsilon$ sufficiently small.
\end{proof}

\begin{figure}
\caption{The relative error for the initial approximation from 
Example~\ref{ex: new}.}
\label{fig: new rel err}
\begin{center}
\includegraphics[scale=0.5]{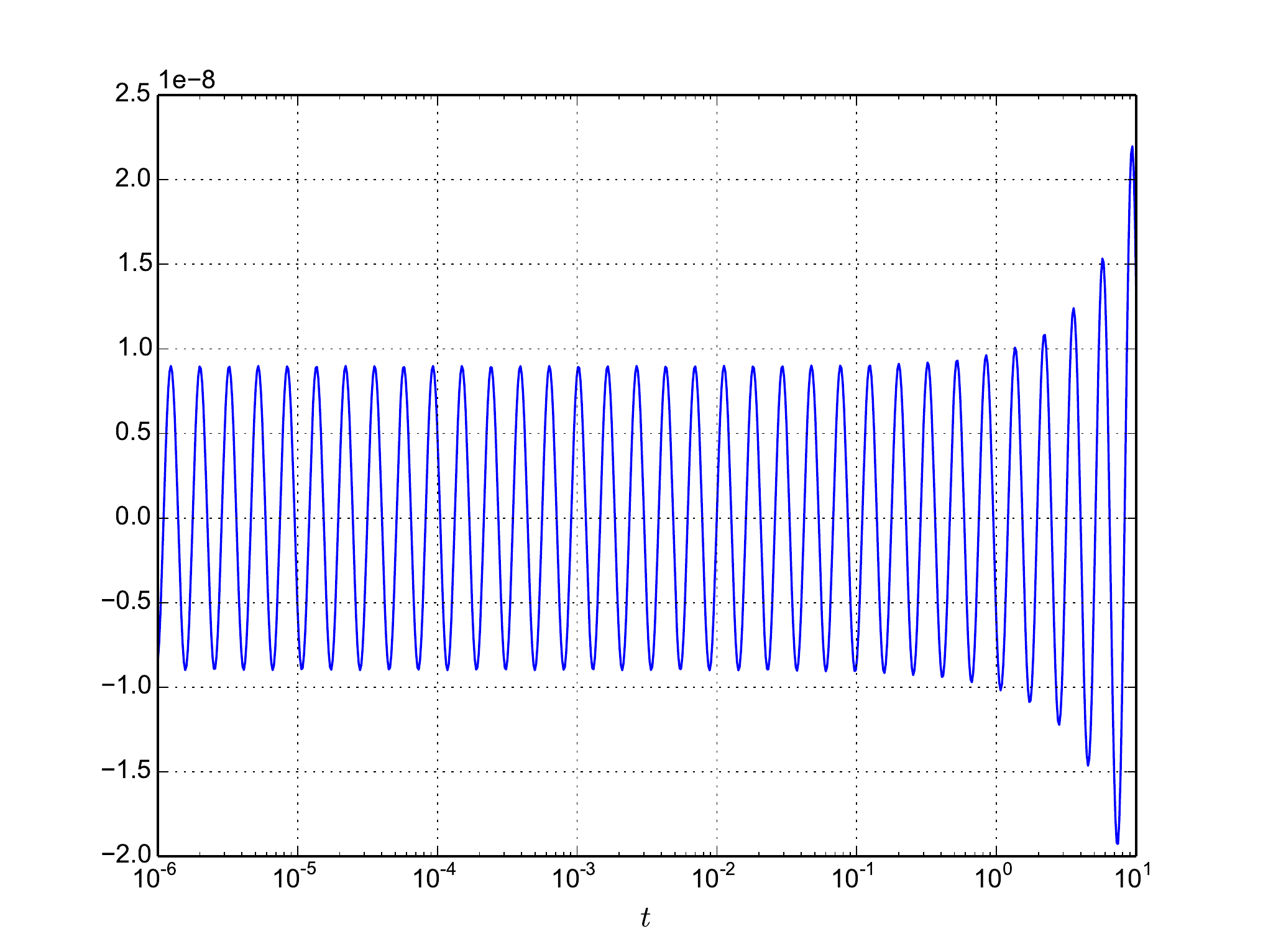}
\end{center}
\end{figure}

Although the error bounds above require $t\in[\delta,1]$, a simple
rescaling allows us to treat a general compact 
subinterval~$[\delta,T]$.  If $\check a_n=a_n/T$~and $\check 
w_n=w_n/T^\beta$, then 
\[
\frac{1}{t^\beta}=\frac{1}{T^\beta}\,\frac{1}{(t/T)^\beta}
	\approx\frac{1}{\Gamma(\beta)}
	\sum_{n=-M}^N\check w_ne^{-\check a_nt}
\]
for $\delta\le t/T\le1$, or in other words 
for~$\delta\cdot T\le t\le T$.  Moreover, the relative error 
$\check\rho^N_M(t)=\rho^N_M(t/T)$ is unchanged by the rescaling.

\begin{example}\label{ex: new}
We took the same values for $\beta$, $\delta$, $T$, $\epsilon$,
$\epsilon_{\text{RD}}$~and $\epsilon_{\text{RT}}$ as in 
Example~\ref{ex: Prony}.  Since the constant~$C_1$ of 
Theorem~\ref{thm: Q(f) error} is difficult to estimate, we again
used~\eqref{eq: rel discr error} to choose $h=0.47962$.  Likewise,
the constant~$C_3$ in Lemma~\ref{lem: trunc} is difficult to 
estimate, so we chose $N=\lceil h^{-1}x_{\delta/T}\rceil=40$.  
However, knowing $C_2=2e/\Gamma(\beta+1)$ we easily determined that
$C_2\exp(-\beta e^{Mh})\le\epsilon_{\text{RT}}$ for~$M=8$.
The exponents and weights~\eqref{eq: new exponents and weights} were 
computed for the interval~$[\delta/T,1]$, and then rescaled as above
to create an approximation for the interval~$[\delta,T]$ with 
$M+1+N=49$~terms and a relative error whose magnitude is at worst 
$2.2\times10^{-8}$. 

The behaviour of the relative error~$\rho^N_M(t)$, shown in
Figure~\ref{fig: new rel err}, suggests a modified strategy: construct 
the approximation for~$[\delta,10T]$ but use it only on~$[\delta,T]$.  
We found that doing so required $N=45$, that is, 5 additional terms, 
but resulted in a nearly uniform amplitude for the relative error of 
about $0.97\times10^{-8}$.  Finally, after applying Prony's method 
with $L=17$~and $K=6$ we were able to reduce the number of terms 
from~$M+1+N=54$ to~$43$ without increasing the relative error.
\end{example}

To compare these results with those of Li~\cite{Li2010}, let 
$0<\alpha<1$ and let
$k(t)=t^{\alpha-1}/\Gamma(\alpha)$ denote the kernel for the 
fractional integral of order~$\alpha$.  Taking $\beta=1-\alpha$ we 
compute the weights~$w_l$ and exponents~$a_l$ as above and define 
\[
k_M^N(t)=\frac{1}{\Gamma(\alpha)\Gamma(1-\alpha)}\sum_{n=-M}^N
	w_ne^{-a_nt}\quad\text{for $\delta\le t\le T$.}
\]
The fast algorithm evaluates
\[
(\mathcal{K}_M^NU)^n=\int_0^{t_{n-1}} k^N_M(t-s)\tilde U(s)\,ds
	+\int_{t_{n-1}}^{t_n}k(t_n-s)\tilde U(s)\,ds
\]
and our bound $|\rho^N_M(t)|\le\epsilon$ implies 
that $|k_M^N(t)-k(t)|\le\epsilon t^{\alpha-1}/\Gamma(\alpha)$
for $\delta\le t\le T$, so
\[
\bigl|(\mathcal{K}_M^NU)^n-(\mathcal{K}\tilde U)(t_n)\bigr|
	\le\epsilon\int_0^{t_{n-1}}
		\frac{(t_n-s)^{\alpha-1}}{\Gamma(\alpha)}\,|\tilde U(s)|\,ds
	\le\frac{\epsilon t_n^\alpha}{\Gamma(\alpha+1)}\,
		\max_{1\le j\le n}|U^j|,
\]
provided $\Delta t_n\ge\delta$~and $t_n\le T$.  Similarly, the method 
of Li yields $(\mathcal{K}_QU)^n$ but with a bound for 
the \emph{absolute} error in~\eqref{eq: Li quad}, so that
$|k_Q(t)-k(t)|\le\epsilon'$ for~$\delta'\le t<\infty$.  Thus,
\[
\bigl|(\mathcal{K}_QU)^n-(\mathcal{K}\tilde U)(t_n)\bigr|
	\le\epsilon'\,\frac{\sin\pi\alpha}{\pi}\int_0^{t_{n-1}}
		\,|\tilde U(s)|\,ds
	\le\epsilon't_n\,\frac{\sin\pi\alpha}{\pi}\,
	\max_{1\le j\le n}|U^j|,
\]
provided $\Delta t_n\ge\delta$.  Li~\cite[Fig.~3~(d)]{Li2010} 
required about $Q=250$~points to achieve an (absolute) 
error~$\epsilon'\le10^{-6}$ for~$t\ge\delta'=10^{-4}$ 
when~$\alpha=1/4$ (corresponding to~$\beta=1-\alpha=3/4$).  In 
Examples 1~and 2, our methods give a smaller 
error~$\epsilon\le10^{-8}$ using only $M+1+N=43$~terms with a
less restrictive lower bound for the time step, $\delta=10^{-6}$.  
Against these advantages, the method of Li permits arbitrarily 
large~$t_n$.
\section{Conclusion}
Comparing Examples \ref{ex: Prony}~and \ref{ex: new}, we see 
that, for comparable accuracy, the approximation based on the second 
substitution results in far fewer terms because we are able to use a 
much smaller choice of~$M$.  However, after applying Prony's method 
both approximations are about equally efficient.  If Prony's method 
is not used, then the second approximation is clearly superior. 
Another consideration is that the first approximation has more 
explicit error bounds so we can more easily determine suitable choices 
of $h$, $M$~and $N$ to achieve a desired accuracy.

\end{document}